\documentclass[11pt,twoside,english]{article}
\usepackage{amsmath}
\usepackage{amsfonts}
\usepackage{mathrsfs}
\usepackage{color}
\usepackage{ amsmath, amsfonts, amssymb, amsthm, amscd}
\usepackage[T1]{fontenc}
\usepackage[english]{babel}
\usepackage[noinfoline]{imsart}

\newtheorem{theorem}{Theorem}
\newtheorem{lemma}{Lemma}
\newtheorem{proposition}{Proposition}

\newtheorem{definition}{Definition}
\newtheorem{corollary}{Corollary}
\newtheorem{remark}{Remark}

\newcommand{\be}{\begin{equation}}
\newcommand{\ee}{\end{equation}}
\newcommand{\beq}{\begin{eqnarray}}
\newcommand{\eeq}{\end{eqnarray}}

\newcommand{\ced}{\end{proof}}

\begin{document}

\begin{frontmatter}

\title{The Neumann Boundary Problem for Elliptic Partial Differential Equations with Nonlinear Divergence Terms}
\date{}
\runtitle{}
\author{\fnms{Xue}
 \snm{YANG}\corref{}\ead[label=e1]{xyang2013@tju.edu.cn}}
\address{Tianjin University
\\\printead{e1}}
\author{\fnms{Jing}
 \snm{ZHANG}\corref{}\ead[label=e2]{zhang\_jing@fudan.edu.cn}}
\address{Fudan University
\\\printead{e2}}

\runauthor{X. Yang and J. Zhang}

\begin{abstract}
We prove the existence and uniqueness of weak solution of a Neumann boundary problem for an elliptic partial differential equation (PDE for short)  with a singular divergence term which can only be understood in a weak sense. A probabilistic approach is applied by studying the backward stochastic differential equation (BSDE for short) corresponding to the PDE. 
\end{abstract}

\end{frontmatter}

\section{Introducton }
In this paper, our aim is to use the probabilistic method to solve the Neumann boundary value problem for semilinear second order elliptic PDE of the following form:
\begin{equation}
\label{PDE}
\left\{
\begin{split}
&\frac{1}{2}\Delta u+ \langle b, \nabla u\rangle+qu-div(g(\cdot,u,\nabla u))+f(\cdot,u,\nabla u)=0,\quad \mbox{on} \  D,\\
&\langle \nabla u-2 g(\cdot,u,\nabla u), \vec{n}\rangle+h(\cdot,u)=0,\quad \mbox{on}\   \partial D,
\end{split}
\right.
\end{equation}
where $D$ is a bounded domain in $\mathbb{R}^N$ with smooth boundary. $\vec n$ is the unit inward normal vector field of $D$ on the boundary $\partial D$. $f$, $g$ and $h$ are nonlinear measurable functions. $b=(b_1,\cdots,b_N)$ is a measurable $\mathbb{R}^N$-valued function on $D$ and $q$ is a measurable $\mathbb{R}$-valued function on $D$. Since $g$ is not differentiable,  the singular divergence term $''div g''$ involved in the equation will be understood as a distribution, and a classical Sobolev weak solution will be considered.

\vspace{2mm}
The probabilistic approaches to the boundary value problems for second order differential operators have been adopted by many authors and the earliest work went back to 1944 (see $\cite{K}$). There have been extensive studies on the Dirichlet boundary problems (see \cite{BH}, \cite{Gernard}, \cite{CZ}, \cite{DP}, \cite{PEI} and \cite{Z}). However, to our knowledge, there are few articles on the Neumann boundary problems.

\vspace{2mm}
When $b=0$, $g=0$ and $f=0$, the following linear Neumann boundary problem
\begin{equation*}
\left\{\begin{split}
&\frac{1}{2}\Delta u(x)+qu(x)=0 \quad\textrm{on $D$}\\
&\frac{1}{2}\frac{\partial u}{\partial \vec{n}}(x)=\phi(x) \quad\textrm{on $\partial D$ }
\end{split}\right.
\end{equation*}
was solved both in $\cite{BH}$ and $\cite{PEI}$. The solution was given in the following representation:
\begin{equation*}
u(x)=E^x[\int_{0}^{\infty}e^{\int_{0}^{t}q(B_{s})ds}\phi(B_{t})dL^{0}_{t}]\,,
\end{equation*}
where $(B_{t})_{t\geq0}$ is the reflecting Brownian motion in the domain $D$ associated with the infinitesimal generator $\frac{1}{2}\Delta$,
the process $\{L^{0}_{t}\}_{t\geq0}$ is the boundary local time expressed as $L^{0}_{t}=\int_{0}^{t}I_{\partial D}(B_{s})dL^0_s$.

\vspace{2mm}
In the case that a divergence term $div(g)$ is involved, even if $g$ is simply  independent of $\nabla u$ and in a linear form, i.e.  $ g( x,u,\nabla u) =\hat{B}(x)u(x)$,  where $\hat B$ is an integrable vector value function, it was found  that the generator $\mathcal{L}=\frac{1}{2}\Delta+b\cdot\nabla+div(\hat B\cdot)+q$ would not associate with any Markov processes in general. And the term $div(\hat{B}\cdot)$ can not be handled by Girsanov  or Feyman-Kac transformation directly either.  The analysis of the boundary value problems for this general operator $\mathcal{L}$ in the PDEs' literature (see e.g. \cite{GT}, \cite{T}) was always established under an extra condition:
$$-div(\hat{B}\cdot)+q\leq 0$$
in the sense of distribution. So that the maximum principle could be used.

\vspace{2mm}
The probabilistic approach for studying problem \eqref{PDE} in the case of $f=0$ and the linear divergence term $div(\hat{B}\cdot)$ was applied in \cite{CZ2} (see also \cite{CZ} for the corresponding Dirichlet boundary problem). The term $div(\hat{B}\cdot)$ was tackled by using the time-reversal Girsanov transform of the symmetric reflecting diffusion associated with the operator $\frac{1}{2}\nabla\cdot(A\nabla)$.

\vspace{2mm}
 In \cite{YZ},  problem $(\ref{PDE})$ with nonlinear $f(x,u,\nabla u)$  was studied, which generalized the result in \cite{CZ2}.  But the divergence term still had to be linear, since the  strategy of time-reversal and h-transform was used to transfer the operator $\mathcal{L}$ with divergence term to $\mathcal{L}_2=\frac{1}{2}\Delta+B\cdot\nabla+Q$, so that the Girsanov and Feyman-Kac transform can be applied in the further calculations.


\vspace{2mm}
The direct motivation of this article is to generalize the result in \cite{YZ}. When nonlinear divergence term $div (g(\cdot, u,\nabla u))$  is considered , it can not be treated as a part of the generator operator because of the nonlinearity. In order to deal with this singular term,  inspired by the method introduced in $\cite{S}$,  we consider firstly the case of $g$ independent of $u$ and $\nabla u$, then substitute the divergence term by $div(\nabla G)$ in the weak sense, where $G$ is a function in Dirichlet space so that the decomposition and calculus can be carried out in the framework of Dirichlet forms.  This method also allows us to give the solution a probabilistic interpretation in the form of a BSDE with forward-backward stochastic integration, which produces a candidate for the solution. The PDE with nonlinear divergence can be solved by Picard iteration with both analytic and probabilistic methods independently. 

\vspace{2mm}
Due to the probabilistic method that is applied in this paper to solve problem \eqref{PDE}, it turns out that we need to prove the existence and uniqueness of solution for the BSDE which connect to the PDE. The study on this kind of BSDEs,  with infinite time horizon, forward-backward stochastic integration and local time integration,  is actually of independent interest.

\vspace{2mm}
The article is organized as follows. In the second section we set notations and recall the decomposition of reflecting diffusions. Then the probabilistic interpretation of the divergence term is given in Section 3. In Section 4 we prove the existence and uniqueness of the solution for the PDEs with linear divergence terms. The last section is devoted to using both analytic and probabilistic methods to  solve the PDEs with nonlinear divergence terms by Picard iteration.

\section{Preliminaries}
The domain $D\subset\mathbb{R}^N$ is bounded with smooth boundary and we assume there is a smooth function $\psi$ such that
$$D=\{x\in\mathbb{R}^N| \psi(x)>0\}\qquad\mbox{and}\qquad \partial D=\{x\in\mathbb{R}^N|\psi(x)=0\}.$$
On $\partial D$, $\vec{n}:=\nabla \psi$ coincides with the unit vector pointing inward the interior of $D$.
Set function $d(x):=d(x,\bar{D})^2$ in a neighborhood of $\bar{D}$ , then $d(x)=0$ in $\bar{D}$ and $d(x)>0$ otherwise. 
The penalization term $\vec\delta(x):=\nabla d(x)$ satisfies $\langle \nabla \psi (x), \vec\delta(x)\rangle \leq 0$, for all $x\in\mathbb{R}^N$.  
Let $dx$ denote the Lebesgue measure on $\mathbb{R}^N$ and $d\sigma$ the $(N-1)$-dimensional Lebesgue measure on $\partial D$.
\vspace{2mm}

Let $L^2(D)$ be the space of square integrable functions on $D$ with the inner product and norm as follows

$$(f,g):=\int_D f(x)g(x)dx\qquad \mbox{and}\qquad \|f\|^2:=(f,f).$$

The first order Sobolev space on $D$ is denoted by $H^{1}(D)$:
$$H^{1}(D)=\{ f\in L^2(D)| \nabla f\in L^2(D)\}.$$

\vspace{3mm}
Suppose the measurable functions
\begin{eqnarray*}
f:\  \mathbb{R}^N\times \mathbb{R}\times \mathbb{R}^N\rightarrow \mathbb{R};\ \  h:\  \mathbb{R}^N\times \mathbb{R}\rightarrow \mathbb{R};\ \  g=(g_1,\cdots,g_N):\  \mathbb{R}^N\times \mathbb{R}\times \mathbb{R}^N\rightarrow \mathbb{R}^N
\end{eqnarray*}
satisfy the following conditions:  there are positive constants $\alpha,\beta, K, M, k,\beta'$ satisfying $K^2<2\alpha$, such that, for any $y,y'\in\mathbb{R}$, $z, z'\in\mathbb{R}^N$,

\vspace{2mm}
{\bf(i)} $(y-y')(f(x,y,z)-f(x,y',z))\leq -\alpha |y-y'|^2$;\\
{\bf(ii)} $(y-y')(h(x,y)-h(x,y'))\leq -\beta |y-y'|^2$;\\
{\bf(iii)} $|f(x,y,z)-f(x,y,z')|\leq K|z-z'|$;\\
{\bf(iv)} $f(x,y,z)$ and $h(x,y)$ is continuous with respect to $y$ for all $x\in D,z\in\mathbb{R}^N$, a.e.;\\
{\bf(v)} $|f(x,y,z)|\leq M(1+|y|+|z|)$, $|h(x,y)|\leq M(1+|y|)$ and $|g(x,y,z)|\leq M$;\\
{\bf(vi)} $|g(x,y,z)-g(x,y',z')|\leq k(|y-y'|+|z-z'|)$; \\
{\bf(vii)} $|h(x,y)-h(x,y')|\leq \beta' |y-y'|$.

\vspace{3mm}
Consider the operator
\begin{eqnarray*}
\mathcal{L}_{1}=\frac{1}{2}\sum_{i,j=1}^N \frac{\partial^2}{\partial x_ix_j}+\sum_{i=1}^N  b_i(x)\frac{\partial}{\partial x_i}
\end{eqnarray*}
on the domain D equipped with the Neumann boundary condition:
\begin{eqnarray*}
\frac{\partial}{\partial \vec{n}}:=\langle \vec{n},\nabla\cdot\rangle=0,\quad \textrm{on} \quad \partial D.
\end{eqnarray*}

Let $b:\mathbb{R}^N\rightarrow \mathbb{R}^N$  be uniformly bounded and satisfy Lipschitz condition, i.e. there exists a constant $C_0>0$, such that for any $x,x'\in \mathbb{R}^N$,
$$|b(x)-b(x')|\leq C_0 |x-x'|.$$

It is well known that there is a unique reflecting diffusion process $(\Omega, \mathcal{F}_{t}, X^{x}(t), P^x, x\in D)$ associated with the generator $\mathcal{L}_{1}$ (see \cite{LS}).


Let $E^x$ denote the expectation under the probability  measure $P^x$.

Then the process $X^{x}(t)$ has the following decomposition:
\begin{eqnarray}
\label{deomposition of X}
X^{x}(t)=X^{x}(0)+M^{x}(t)+\int_{0}^{t}{b}(X^{x}(s))ds+\int_{0}^{t}\vec{n}(X^{x}(s))dL_{s},\quad P^{x}-a.s..
\end{eqnarray}

Here, $M^{x}(t)$ is a $\mathcal{F}_{t}-$measurable square integrable continuous martingale additive functional. 
$L_t$ is a positive increasing continuous additive functional which is expressed as $L_{t}=\int_{0}^{t}I_{\{X^{x}(s)\in\partial D\}}dL_s$. 
$\{L_t, t\geq 0\}$ is called the boundary local time of $X$.

In the following discussion, we write $X^x(t)$ as $X_t$ or $X(t)$ for simplicity.

\vspace{3mm}
We assume the measurable function $q:\mathbb{R}^N\rightarrow \mathbb{R}$, $q\in L^{p}(D)$, for $p>\frac{N}{2}$, satisfying the following conditions:

\textbf{(C.1)} there exists $x_0\in D$ such that
$$E^{x_0}[\int_0^\infty e^{\int_0^tq(X_s)ds}dL_t]<+\infty\,;$$

\textbf{(C.2)} there exists $x_1\in D$ such that
$$E^{x_1}[\int_0^\infty e^{2\int_0^t q(X_s)ds}dL_t]<+\infty\,;$$

\textbf{(C.3)}  
$$
\sup_{x\in D}E^x[\int_0^\infty e^{2\int_0^t q(X_s)ds} |q(X_t)|^2dt ]<\infty\,.
$$

\vspace{3mm}
Finally, we give the definition of the solution in which we are interested in this article.
\begin{definition}
	A function $u\in H^{1}(D)$ is said to be a weak solution of PDE $(\ref{PDE})$ if for any test function $\phi\in C^{\infty}(D)$,
	\begin{equation}\label{weaksolution}\begin{split}
	&\frac{1}{2}\int_D \langle \nabla u,\nabla \phi\rangle(x)dx-\int_D \langle b,\nabla u\rangle(x)\phi(x)dx+\int_D q(x)u(x)\phi(x)dx\\
	=&\int_D f(x,u,\nabla u)\phi(x)dx+\int_{D}\langle g(\cdot,u,\nabla u),\nabla \phi\rangle(x)dx-\frac{1}{2}\int_{\partial D} h(x)\phi(x)d\sigma(x).
	\end{split}\end{equation}
\end{definition}

\section{Interpretation of the Divergence Term}
In this section, we will give a stochastic representation for the divergence term in $(\ref{PDE})$ which can be expressed as a measurable field. The second order operator in  $(\ref{PDE})$ is nonsymmetric and associated with a reflecting diffusion.

The bilinear form
 $$\mathcal{E}(u,v)=\frac{1}{2}\int_D \sum_{i=1}\frac{\partial u}{\partial x_i}\frac{\partial v}{\partial x_i}dx,\quad \forall u,v\in H^1(D)$$
is associated with the generator $L_0=\frac{1}{2}\Delta$ satisfying the Neumann boundary condition $\frac{\partial u}{\partial\vec{n}}=0$ on $\partial D$. Set the operator $Lu:=L_0 u+\langle b,\nabla u\rangle$. Then $L$ generates a semigroup $(P_t)_{t\geq 0}$ which possesses continuous densities $\{p(t,x,y),t\geq 0, x,y\in \bar{D}\}$.  It is well known that the reflecting diffusion \eqref{deomposition of X} is associated with operator $L$,
and for any $u\in H^1(D)$, the Fukushima decomposition is as follows
$$
u(X_t)-u(X_s)=Mu|^t_s+Nu|^t_s\,,
$$
where $Mu|^t_s:=\int_s^t \langle \nabla u(X_r),dB_r\rangle$ is the martingale additive functional and $Nu|^t_s$ is the zero-energy additive functional. 
For $u\in C(\bar{D})$,
$$
Nu|^t_s:=\int_s^t Lu(X_r)dr+\int_s^t \frac{\partial u}{\partial\vec{n}}(X_r)dL_r\,,
$$
where $L_t$ is the additive functional corresponding to the Lebesgue measure $\sigma(x)$ on $\partial D$. It follows that
$$
E^x[\int_{0}^{t}f(X_r)dL_r]=\int_{0}^{t}\int_{\partial{D}} p(r,x,y)f(y)\sigma(dy)dr.
$$

Consider the reverse process $(X_{T-t})_{t\in[0,T]}$ under the probability $P^{o}$, for $o\in\bar{D}$, with the non-homogenous transition function
\begin{eqnarray*}
Q_{0,t}u(x)=\frac{\int_D p(T-t,o,y)u(y)p(t,y,x)dy}{p(T,o,x)}.
\end{eqnarray*}
We denote the density of $Q_{0,t}$ by $p_Q(t,x,y)=\frac{p(T-t,o,y)p(t,y,x)}{p(T,o,x)}$.
\begin{lemma}
Fix $o\in \bar{D}$ and set $p_t(x)=p(t,o,x)$,
\begin{equation*}\begin{split}
Q_{0,t}u-u=\int_{0}^{t}Q_{0,r}(\frac{1}{2}\Delta u-\langle b,\nabla u\rangle
&+\frac{\langle \nabla p_{T-r},\nabla u\rangle}{p_{T-r}})dr\\&+\frac{1}{2}\int_{0}^{t}\int_{\partial{D}}p_Q(r,x,y)\frac{\partial u}{\partial \vec n}(y)\sigma(dy)dr.
\end{split}\end{equation*}
\end{lemma}

\begin{proof}
\begin{equation*}\begin{split}
&p_T(x)\int_{0}^{t}Q_{0,r}(\frac{1}{2}\Delta u-\langle b,\nabla u\rangle)dr\\
=&\int_{0}^{t}\int_D p(T-r,o,y)(\frac{1}{2}\Delta u-\langle b,\nabla u\rangle)(y)p(r,y,x)dy\\
=&-\int_{0}^{t}\int_D L^*p_{T-r}(y)u(y)p(r,y,x)dydr+\int_{0}^{t}\int_DL^*p(r,y,x)p(T-r,o,y)u(y)dydr\\
&-\int_{0}^{t}\int_D \langle \nabla p_{T-r},\nabla u\rangle p(r,y,x)dydr
-\frac{1}{2}\int_0^t\int_{\partial D}p(T,o,y)\frac{\partial u}{\partial \vec{n}}(y)p(r,y,x)\sigma(dy)dr\\
=& \int_D p(T-t,o,y)u(y)p(t,y,x)dy-p_T(x)u(x)-\int_{0}^{t}\int_D \langle \nabla p_{T-r},\nabla u\rangle p(r,y,x)dydr\\
&-\frac{1}{2}\int_0^t\int_{\partial D}p(T,o,y)\frac{\partial u}{\partial \vec{n}}(y)p(r,y,x)\sigma(dy)dr,
\end{split}\end{equation*}
where the last equality is obtained by $L^* p(t,o,y)=\partial_t p(t,o,y)$.
\end{proof}

\begin{proposition}
Fix $o\in \bar{D}$ and set the following process
\begin{equation*}\begin{split}
\bar{M}u|^T_{T-t}:=u(X_{T-t})&-u(X_T)-\int_{0}^{t}(\frac{1}{2}\Delta u-\langle b,\nabla u\rangle)(X_{T-r})dr
\\&-\int_{0}^{t}\frac{\langle \nabla p_{T-r},\nabla u\rangle(X_{T-r})}{p_{T-r}(X_{T-r})}dr-\int_{T-t}^{T}\frac{\partial u}{\partial \vec n}(X_r)dL_r.
\end{split}\end{equation*}
\vspace{2mm}

(1). $\{\bar{M}u|^T_{T-t}\}_{t\in[0,T]}$ is a martingale with respect of the filtration $\mathcal{F}^{'}_t=\sigma\{X_{T-s},s\in[0,t]\}$
and $$\bar{M}u|^T_t-\bar{M}u|^T_s=\bar{M}u|^s_{t}.$$
(2). The following relation holds:
$$u(X_t)-u(X_0)=\frac{1}{2}Mu|^t_0-\frac{1}{2}(\bar{M}u|^T_0-\bar{M}u|^T_t)+\int_{0}^{t}\langle b,\nabla u\rangle (X_r)dr
-\frac{1}{2}\int_{0}^{t}\frac{\langle \nabla p_{r},\nabla u\rangle(X_{r})}{p_{r}(X_{r})}dr.$$
\end{proposition}

\begin{proof}
Since
\begin{equation*}\begin{split}
\bar{M}u|^T_{t}:=u(X_{t})-u(X_T)&-\int_{0}^{T-t}(\frac{1}{2}\Delta u-\langle b,\nabla u\rangle)(X_{T-r})dr
\\&-\int_{0}^{T-t}\frac{\langle \nabla p_{T-r},\nabla u\rangle(X_{T-r})}{p_{T-r}(X_{T-r})}dr-\int_{t}^{T}\frac{\partial u}{\partial \vec n}(X_r)dL_r, 
\end{split}
\end{equation*}
it follows that 
\begin{equation*}\begin{split}
\bar{M}u|^T_{t}-\bar{M}u|^T_{s}&=u(X_{t})-u(X_{s})-\int_0^{s-t} (\frac{1}{2}\Delta u-\langle b,\nabla u\rangle)(X_{s-r})dr\\
&-\int_{0}^{s-t}\frac{\langle \nabla p_{s-r},\nabla u\rangle(X_{s-r})}{p_{s-r}(X_{s-r})}dr
-\int_{t}^{s}\frac{\partial u}{\partial \vec n}(X_{r})dL_r=\bar{M}u|^s_t
\end{split}
\end{equation*}
and
\begin{equation*}\begin{split}
u(X_t)-u(X_0)=&\,\bar{M}u|^T_t-\bar{M}u|^T_0-\int_{T-t}^{T}(\frac{1}{2}\Delta u-\langle b,\nabla u\rangle)(X_{T-r})dr
\\&-\int_{0}^{t}\frac{\partial u}{\partial \vec n}(X_{r})d{L}_r
-\int_{T-t}^{T}\frac{\langle \nabla p_{T-r},\nabla u\rangle(X_{T-r})}{p_{T-r}(X_{T-r})}dr\\
=&-\bar{M}u|^t_0-\int_{0}^{t}(\frac{1}{2}\Delta u-\langle b,\nabla u\rangle)(X_{r})dr
\\&-\int_{0}^{t}\frac{\partial u}{\partial \vec n}(X_{r})d{L}_r-\int_{0}^{t}\frac{\langle \nabla p_{r},\nabla u\rangle(X_{r})}{p_{r}(X_{r})}dr.
\end{split}
\end{equation*}
Then
\begin{equation*}
2(u(X_t)-u(X_0))=Mu|^t_0-\bar{M}u|^t_0+2\int_{0}^{t}\langle b,\nabla u\rangle (X_r)dr
-\int_{0}^{t}\frac{\langle \nabla p_{r},\nabla u\rangle(X_{r})}{p_{r}(X_{r})}dr.
\end{equation*}

Therefore, we get the forward-backward martingale decomposition
$$u(X_t)-u(X_0)=\frac{1}{2}Mu|^t_0-\frac{1}{2}\bar{M}u|^t_0+\int_{0}^{t}\langle b,\nabla u\rangle (X_r)dr
-\frac{1}{2}\int_{0}^{t}\frac{\langle \nabla p_{r},\nabla u\rangle(X_{r})}{p_{r}(X_{r})}dr.$$
\end{proof}

\begin{corollary}
(1). For $u,v\in H^1(D)$,
$$\langle Mu, Mv\rangle_t=\int_{0}^{t}\langle \nabla u, \nabla v\rangle(X_r)dr$$ and 
$$\langle \bar{M}u|^T_\cdot, \bar{M}v|^T_\cdot \rangle_t=\int_{t}^{T}\langle \nabla u, \nabla v\rangle(X_r)dr.$$
(2). For $x=(x_1,\cdots,x_N)\in D$, set $u_i (x)=x_i$, $M^i(t)=Mu_i|^t_0$ and $\bar{M}^i(t,T)=\bar{M}u_i|^T_t$, then
$$X^i_t-X^i_0=\frac{1}{2}M^i(t)-\frac{1}{2}\bar{M}^i(0,t)+\int_{0}^{t} b_i(X_r)dr
-\frac{1}{2}\int_{0}^{t}\frac{\partial_i p(X_{r})}{p_{r}(X_{r})}dr.$$
\end{corollary}

\vspace{4mm}
For $g=(g_1,\cdots,g_N): \mathbb{R}^N\rightarrow \mathbb{R}^N$, we define the backward stochastic integral 
\begin{equation*}
\int_s^t g_i(X_{r})d\bar{M}^i_t:=(L^2-)\lim_{\delta\rightarrow 0}\sum_{j=0}^{n-1} g(X_{t_{j+1}})\bar{M}^i(t_j,t_{j+1}),
\end{equation*}
where the limit is over the partition $s=t_0<t_1<\cdots<t_n=t$ and $\delta=\max_j (t_{j+1}-t_j)$.

Define
\begin{equation}
\int_s^t g\ast dX_r=\int_s^t g(X_r) dM_r+\int_s^t g(X_r)d\bar{M}_r
+\int_s^t \frac{\langle g, \nabla p_r\rangle}{p_r}(X_r)dr+2\int_s^t \langle g,\vec n\rangle(X_r)dL_r\,.
\end{equation}

\begin{proposition}\label{decompositionofG}
For $G\in H^1(D)$, then we have the decomposition
\begin{equation}
G(X_t)-G(X_s)=\int_s^t \langle \nabla G(X_r),dM_r\rangle+\int_s^t \langle b,\nabla G\rangle(X_r)dr+\int_s^t \frac{\partial G}{\partial \vec n}(X_r)dL_r-\frac{1}{2}\int_s^t \nabla G\,\ast\,dX_r.
\end{equation}
\end{proposition}

The following lemma is very important in the interpretation of the divergence term $div g$ in PDE $(\ref{PDE})$.
\begin{lemma}
\label{div g}
For $g\in L^2(\mathbb{R}^N;\mathbb{R}^N)$, if there is a function $G\in L^2(\mathbb{R}^N)$, such that $div g=G$ in weak sense, then
\begin{equation*}
\int_s^t G(X_r)dr=-\int_s^t g\ast dX_r.
\end{equation*}
\end{lemma}

\section{PDEs with Linear Divergence Terms}\label{linear.case}

In this section, we will prove the existence and uniqueness of solution for the following Neumann boundary problem with linear divergence term, 
i.e. $g(x,u,\nabla u)=g(x)$,
\begin{equation}
\label{linear PDE}
\left\{
\begin{split}
&\frac{1}{2}\Delta u+ \langle b, \nabla u\rangle+qu-div(g)+f(\cdot, u,\nabla u)=0,\quad \mbox{on} \  D\\
&\langle \nabla u-2 g, \vec{n}\rangle+h(\cdot,u)=0,\quad \mbox{on}\   \partial D.
\end{split}
\right.
\end{equation}
Furthermore, the probability interpretation of the solution will also be established.

%

\vspace{2mm}
The following analytic result will be used in the later discussion (see Chapter 8 in \cite{GT}). 
\begin{proposition}
\label{regularity G}
For $g \in L^q(\mathcal{O})$, where $\mathcal{O}\subset\mathbb{R}^N$ is bounded and $q>N$,  there exists a unique weak solution $G\in H^1_0(\mathcal{O})$ for the following equation
\begin{eqnarray*}
\Delta G-G=div(g).
\end{eqnarray*}
Furthermore, $G$ is uniformly bounded, i.e. $\sup\limits_{\mathcal{O}}|G|\leq C\|g\|_{L^q}$, where $C=C(N,q,|\mathcal{O}|)$.\\
If we suppose $g\in L^\infty(\mathcal{O})$ and $\mathcal{O}$ is a $C^{1,1}$-domain, then $G\in C^{1,1}(\bar{\mathcal{O}})$, i.e. there exists a constant $C>0$, for any $x,x'\in \bar{\mathcal{O}}$, $|G(x)-G(x')|\leq C |x-x'|$.
\end{proposition}

\begin{remark}\label{relationGg}
Since $g\in L^2(D)$, for any $\psi\in C^\infty(D)$,  
$$|\mathcal{S}(\psi)|:=|\int_D\langle g,\nabla \psi\rangle dx|\leq \|g\|\cdot\|\psi\|_{H^1}$$
implies that $\mathcal{S}$ is a bounded linear operator on $H^1(D)$. By Riesz representation theorem, there is function $G\in H^1(D)$ such that 
$$\int_D\langle g,\nabla \psi\rangle dx=\int_D\langle \nabla G,\nabla\psi\rangle(x)+ G(x)\psi(x)ds.$$

Suppose $g\in L^\infty (D)$,  we can find a bounded domain $\mathcal{O}$ with smooth boundary, such that $D\subset\subset \mathcal{O}$ and extend $g$ on $\mathcal{O}$ such that $g\in L^\infty(\mathcal{O})$. Therefore, by Proposition \ref{regularity G}, there exists a H\"{o}lder continuous function $\bar G\in H^1_0(\mathcal{O})$, for any test function  $\phi\in C_0^\infty(\mathcal{O})$,
\begin{eqnarray*}
	\int_\mathcal{O} \langle g,\nabla\phi\rangle(x)dx=\int_\mathcal{O} \langle \nabla \bar G,\nabla\phi\rangle(x)+ \bar G(x)\phi(x)dx.
\end{eqnarray*}
By the uniqueness of Reisz representation theorem, we find $\bar G$ restricted on $D$  denoted by $\bar G|_D$ in $H^1(D)$, therefore $\bar G|_D=G$  implies that $G\in C^{1,1}(D)$.
\end{remark}
\begin{remark}
By Remark \ref{relationGg}, PDE ($\ref{linear PDE}$) is equivalent to the following equation
\begin{equation}
\label{tran. linear PDE 1}
\left\{
\begin{split}
&\frac{1}{2}\Delta (u-2G)+ \langle b, \nabla u\rangle+qu+G+f(\cdot,u,\nabla u)=0,\quad \mbox{on} \  D,\\
&\langle \nabla (u-2G), \vec{n}\rangle+h(\cdot,u)=0,\quad \mbox{on}\   \partial D.
\end{split}
\right.
\end{equation}
If we set $\tilde u(x)=u(x)-2G(x)$, then equation $(\ref{tran. linear PDE 1})$ can be rewritten as 
\begin{equation}
\label{tran. linear PDE}
\left\{
\begin{split}
&\frac{1}{2}\Delta \tilde u+ \langle b, \nabla \tilde u\rangle+q\tilde u+\tilde f(\cdot,\tilde u,\nabla \tilde u)=0,\quad \mbox{on} \  D,\\
&\langle \nabla \tilde u, \vec{n}\rangle+  \tilde h(\cdot,\tilde u)=0,\quad \mbox{on}\   \partial D,
\end{split}
\right.
\end{equation}
with $$\tilde f(x,y,z)=2\langle b(x), \nabla G(x)\rangle+2q(x)G(x)+G(x)+f(x, y+2G(x), z+2 \nabla G(x))$$ and $$ \tilde h(x, y)=h(x, y+2G(x)).$$
\end{remark}



\begin{proposition}
Under conditions \textbf{(i)-(v)} and \textbf{(C.2)}, \textbf{(C.3)},
assume there exist two negative constants $\lambda$ and $\mu$ such that $-2\alpha+K^2<\lambda<0$ and $-2\beta<\mu<0$, 
 then the following BSDE admits a unique solution $(Y^x, Z^x)$,
\begin{equation}
\label{bsde}
\left\{\begin{split}
&Y^x_t=Y^x_T+\int_t^T q(X_s)Y^x_sds+\int_t^T F(s,Y^x_s, Z^x_s)ds+\int_t^T H(s,Y^x_s)dL_s-\int_t^T\langle Z^x_s,dM_s\rangle, \\
&\lim\limits_{t\rightarrow\infty} e^{\frac{1}{2}(\lambda t+\mu L_t)+\int_0^tq(X_s)ds}Y^x_t=0, \quad \mbox{in}\ \ L^2(P^x),
\end{split}\right.
\end{equation}
where $F(t,y,z)=\tilde{f}(X_t, y, z)$ and $H(t,y)=\tilde{h}(X_t, y)$.
\end{proposition}

\begin{proof}
Firstly, we prove the existence of solution. 
Set $$\tilde F(t,y,z)=e^{\int_0^tq(X_s)ds}F(t, e^{-\int_0^tq(X_s)ds}y,e^{-\int_0^tq(X_s)ds}z)$$
and $$\tilde H(t,y)=e^{\int_0^tq(X_s)ds}H(t,e^{-\int_0^tq(X_s)ds}y).$$
Then, it is easy to check that 

(i) $(y-y')(\tilde F(t,y,z)- \tilde F(t,y',z))\leq -\alpha |y-y'|^2$;

(ii) $(y-y')(\tilde H(t,y)- \tilde H(t,y'))\leq -\beta |y-y'|^2$;

(iii) $| \tilde F(t,y,z)- \tilde F(t,y,z')|\leq K|z-z'|$.


\vspace{2mm}
Furthermore, by the boundedness of functions $b,G,\nabla G$ and the assumption \textbf{(v)}, there is a constant $C$, such that
\begin{equation*}
| \tilde F(t,y,z)|\leq Ce^{\int_0^tq(X_s)ds}(1+q(X_t))+K(|y|+|z|)
\end{equation*}
and
\begin{equation*}
| \tilde H(t,y)|\leq Ce^{\int_0^tq(X_s)ds}+K|y|.
\end{equation*}

Under \textbf{(C.2)}, from \cite{CZ} and \cite{YZ}, we know that there are two positive constants $C,\theta$ such that
$\sup_{x\in D} E^x[e^{2\int_0^tq(X_s)ds}]<Ce^{-\theta t}$ and $\sup_{x\in D} E^x[\int_0^\infty e^{2\int_0^tq(X_s)ds}dL_t]<+\infty$, then 
$$
\sup_{x\in D} E^x[\int_0^\infty e^{2\int_0^tq(X_s)ds}(dt+dL_t)]<+\infty.
$$
So by the negativity of $\lambda,\mu$ and condition \textbf{(C.3)}  we have 
$$E^x[\int_0^\infty e^{\lambda t+\mu L_t+2\int_0^t q(X_s)ds}((1+|q(X_t)|^2)dt+dL_t)]<+\infty.$$

Therefore, by Theorem 2.1 in \cite{PZ}, there exists a unique pair of solution $(\tilde Y^x,\tilde Z^x )$ for the following BSDE
\begin{equation}
\left\{
\begin{split}
&\tilde Y^x_t=\tilde Y^x_T+\int_t^T \tilde F(s, \tilde Y^x_s, \tilde  Z^x_s)ds+\int_t^T \tilde H(s,\tilde Y^x_s)dL_s-\int_t^T\langle \tilde Z^x_s,dM_s\rangle, & t<T,\\
&\lim_{t\rightarrow \infty} e^{\frac{1}{2}(\lambda t+\mu L_t)}\tilde Y^x_t=0,\quad \mbox{in}\ L^2(P^x).
\end{split}
\right.
\end{equation}

Furthermore,  we have the following estimate
\begin{equation*}\begin{split}
&E^x\Big[\sup_t e^{\lambda t+\mu L_t}|\tilde Y^x_t|^2+\int_0^\infty e^{\lambda t+\mu L_t} |\tilde Z^x_t|^2dt\Big]\\
\leq&\, CE^x\Big[\int_0^\infty e^{\lambda t+\mu L_t} (|\tilde F(t,0,0)|^2dt+|\tilde H(t,0)|^2dL_t)\Big]\\
\leq&\, C'E^x\Big[\int_0^\infty e^{\lambda t+\mu L_t+2\int_0^tq(X_r)dr}((1+|q(X_t)|^2)dt+dL_t)\Big].
\end{split}\end{equation*}
Set $Y^x_t=e^{-\int_0^tq(X_s)ds}\tilde Y^x_t$ and $Z^x_t=e^{-\int_0^tq(X_s)ds}\tilde Z^x_t$. Then It\^o's formula yields
\begin{equation*}\begin{split}
dY^x_t=&-q(X_t)e^{-\int_0^tq(X_s)ds}\tilde Y^x_t-e^{-\int_0^tq(X_s)ds}\tilde F(t, \tilde Y^x_t, \tilde  Z^x_t)dt
-e^{-\int_0^tq(X_s)ds}\tilde H(t,\tilde Y^x_t)dL_t\\
&+e^{-\int_0^tq(X_s)ds}\langle \tilde Z^x_t,dM_t\rangle\\
=&-q(X_t)Y^x_t-F(t, Y^x_t, Z^x_t)dt- H(t, Y^x_t)dL_t+\langle  Z^x_t,dM_t\rangle
\end{split}\end{equation*}
and $\lim\limits_{t\rightarrow \infty} e^{\frac{1}{2}(\lambda t+\mu L_t)+\int_0^t q(X_s)ds}Y^x_t=\lim\limits_{t\rightarrow \infty} e^{\frac{1}{2}(\lambda t+\mu L_t)}\tilde Y^x_t=0$. Moreover, 
\begin{equation*}
\label{bsde estimate}\begin{split}
&E^x\Big[\sup_t e^{\lambda t+\mu L_t+2\int_0^t q(X_s)ds}|Y^x_t|^2+\int_0^\infty e^{\lambda t+\mu L_t+2\int_0^t q(X_s)ds} |Z^x_t|^2dt\Big]\\
&\leq C'E^x\Big[\int_0^\infty e^{\lambda t+\mu L_t+2\int_0^tq(X_s)ds}((1+|q(X_t)|^2)dt+dL_t)\Big].
\end{split}\end{equation*}

Now we turn to prove the uniqueness of solution. 
We assume there exists another pair of solution $(\bar Y^x,\bar Z^x)$ for BSDE ($\ref{bsde}$). Set $\Delta Y_t=\bar Y^x_t-Y^x_t$, $\Delta  Z_t=\bar Z^x_t-Z^x_t$, $\Delta  F_t=F(t,\bar Y^x_t,\bar Z^x_t)-F(t, Y^x_t, Z^x_t)$ and $\Delta  H_t=H(t,\bar Y^x_t)-H(t,Y^x_t)$, then it follows that
\begin{equation*}\begin{split}
&de^{\lambda t+\mu L_t+2\int_0^tq(X_u)du}|\Delta  Y_t|^2\\
=&\,e^{\lambda t+\mu L_t+2\int_0^tq(X_u)du}\big(-2q(X_t)|\Delta  Y_t|^2dt-2\Delta  Y_t \Delta F_t dt-2\Delta  Y_t \Delta H_tdL_t+2\Delta  Y_t\langle \Delta  Z_t,d M_t\rangle\\
&+|\Delta  Z_t|^2dt+(\lambda+2q(X_t))|\Delta  Y_t|^2dt+\mu|\Delta  Y_t|^2dL_t\big).
\end{split}\end{equation*}
For $t<T$,
\begin{equation*}\begin{split}
&e^{\lambda t+\mu L_t+2\int_0^tq(X_s)ds}|\Delta  Y_t|^2+\int_t^Te^{\lambda s+\mu L_s+2\int_0^s q(X_r)dr}|\Delta Z_s|^2ds \\
=&\,e^{\lambda T+\mu L_T+2\int_0^Tq(X_s)ds}|\Delta  Y_T|^2+\int_t^Te^{\lambda s+\mu L_s+2\int_0^s q(X_r)dr}( 2\Delta  Y_s\Delta F_s-\lambda|\Delta  Y_s|^2)ds\\
&+\int_t^Te^{\lambda s+\mu L_s+2\int_0^s q(X_r)dr}( 2\Delta  Y_s\Delta H_s-\mu |\Delta  Y_s|^2)dL_s
-2\int_t^Te^{\lambda s+\mu L_s+2\int_0^s q(X_r)dr} \Delta  Y_s \langle \Delta  Z_s,d M_s\rangle\\
\leq &\,e^{\lambda T+\mu L_T+2\int_0^Tq(X_s)ds}|\Delta  Y_T|^2+\int_t^Te^{\lambda s+\mu L_s+2\int_0^s q(X_r)dr} (-2\alpha+K^2-\lambda)|\Delta  Y_s|^2+|\Delta  Z_s|^2ds\\
&+\int_t^Te^{\lambda s+\mu L_s+2\int_0^s q(X_r)dr} (-2\beta-\mu)|\Delta  Y_s|^2dL_s-2\int_t^Te^{\lambda s+\mu L_s+2\int_0^s q(X_r)dr} \Delta  Y_s \langle \Delta  Z_s,d M_s\rangle.
\end{split}\end{equation*}
Therefore, it follows that
\begin{equation*}\begin{split}
e^{\lambda t+\mu L_t+2\int_0^tq(X_u)du}|\Delta  Y_t|^2
\leq&\, e^{\lambda T+\mu L_T+2\int_0^Tq(X_s)ds}|\Delta  Y_T|^2\\& -2\int_t^Te^{\lambda s+\mu L_s+2\int_0^s q(X_u)du} \Delta  Y_s \langle \Delta  Z_s,d M_s\rangle.
\end{split}\end{equation*}
This implies
$$E^x\Big[e^{\lambda t+\mu L_t+2\int_0^tq(X_s)ds}|\Delta  Y_t|^2\Big]\leq E^x\Big[e^{\lambda T+\mu L_T+2\int_0^Tq(X_s)ds}|\Delta  Y_T|^2\Big].$$
Since $\lim\limits_{t\rightarrow\infty}E^x[e^{\lambda t+\mu L_t+2\int_0^tq(X_u)du}|\Delta  Y_t|^2]=0$ and the arbitrariness of $T$, we find that $\Delta  Y_t=0$, $P^x-a.e.$. The uniqueness is proved.
\end{proof}

\begin{corollary} 
\begin{eqnarray}\label{estimateY0}
\sup_{x\in D}[|Y^x_0|]<+\infty.
\end{eqnarray}
\end{corollary}

\begin{proof}
Estimate $(\ref{bsde estimate})$ yields
\begin{eqnarray*}
|Y^x_0|^2=E^x[|Y^x_0|^2]\leq CE^x\Big[\int_0^\infty  e^{\lambda t+\mu L_t+2\int_0^tq(X_s)ds}((1+|q(X_t)|^2)dt+dL_t)\Big].
\end{eqnarray*}
By \cite{YZ} and the assumptions in the last proposition, 
$$\sup E^x\Big[\int_0^\infty  e^{\lambda t+\mu L_t+2\int_0^tq(X_s)ds}((1+|q(X_t)|^2)dt+dL_t)\Big]<+\infty,$$
 then \eqref{estimateY0} is obtained.
\end{proof}

\begin{theorem}\label{theorem linear}
	Under assumptions \textbf{(i)-(v)} and \textbf{(C.1)-(C.3)},  Neumann problem \eqref{linear PDE} admits a unique bounded weak solution. 
\end{theorem}
\begin{proof}
\textbf{Existence:}
Let $(Y^x,Z^x)$ denote the solution of BSDE $(\ref{bsde})$.
Set $u_0(x)=Y^x_0$ and $v_0(x)=Z^x_0$. Then
$$
u_0(X^x_t)=Y^x_t, \ \ v_0(X^x_t)=Z^x_t.
$$
We consider the following PDE: 
\begin{equation}\label{transformlinear}
\left\{
\begin{split}
&\frac{1}{2}\Delta \tilde u+ \langle b, \nabla \tilde u\rangle+q\tilde u+\tilde{f}(\cdot,u_0,v_0)=0,\quad \mbox{on} \  D;\\
&\langle \nabla \tilde u, \vec{n}\rangle+\tilde h(\cdot,u_0)=0,\quad \mbox{on}\   \partial D,
\end{split}\right.
\end{equation} 
with $\tilde f(x,y,z)=2q(x)G(x)+2\langle b,\nabla G\rangle(x)+G(x)+f(x,y+2G,z+2\nabla G)$ and $\tilde h(x,y)=h(x,y+2G(x))$.
It is known that linear PDE \eqref{transformlinear} has a unique bounded weak solution $\tilde u\in H^1(D)$ (see Theorem 3.1 in \cite{YZ}). 
Next we will prove that $\tilde u=u_0$ and $\nabla \tilde u=v_0$.

We begin with the following decomposition: 
\begin{equation*}\begin{split}
d\tilde u(X_t)=&\langle \nabla\tilde u(X_t),dM_t\rangle-\tilde h(X_t, u_0(X_t))dL_t-q(X_t)\tilde u(X_t)dt-\tilde{f}(X_t,u_0(X_t),v_0(X_t))dt\\
=&\langle \nabla\tilde u(X_t),dM_t\rangle-q(X_t)\tilde u(X_t)dt-H(t, Y^x_t)dL_t-F(t,Y^x_t,Z^x_t).
\end{split}\end{equation*}

Set $\bar u_t=\tilde u(X_t)-u_0(X_t)$ and $\bar z_t=\nabla \tilde u(X_t)-v_0(X_t)$, then
\begin{eqnarray*}
d\bar u_t=-q(X_t)\bar u_tdt+\langle \bar z_t,dM_t\rangle.
\end{eqnarray*}
A simple calculation yields
\begin{eqnarray*}
d e^{\int_0^t q(X_s)ds}\bar u_t=e^{\int_0^t q(X_s)ds}\langle \bar z_t,dM_t\rangle,
\end{eqnarray*}
and for any $t<T$, by \textbf{(C.1)}, we have
\begin{equation*}\begin{split}
0\leq e^{\int_0^t q(X_s)ds}|\bar u_t|&=| E^x[e^{\int_0^T q(X_s)ds}\bar u_T|\mathcal{F}_t]|\\&\leq (\|\tilde u\|_\infty+\|u_0\|_\infty)E^x[e^{\int_0^T q(X_s)ds}|\mathcal{F}_t]\rightarrow 0,\ \ T\rightarrow \infty.
\end{split}\end{equation*}
Therefore, we found $\bar u_t=0$, $\bar z_t=0$. 
This yields
\begin{eqnarray*}
\tilde u(X_t)=u_0(X_t),\ \ \nabla\tilde u(X_t)=v_0(X_t).
\end{eqnarray*}
$\tilde u(x)=E^x[\tilde u(X_0)]=E^x[u_0(X_0)]=u_0(x)$ and $\nabla\tilde u(x)=E^x[\nabla \tilde u(X_0)]=E[v_0(X_0)]=v_0(x)$. It follows that $\tilde u$ is a weak solution of the following PDE
\begin{eqnarray*}
\left\{
\begin{split}
&\frac{1}{2}\Delta \tilde u+ \langle b, \nabla \tilde u\rangle+q\tilde u+\tilde{f}(\cdot,\tilde u,\nabla\tilde u)=0,\quad \mbox{on} \  D\\
&\langle \nabla \tilde u, \vec{n}\rangle+\tilde h(\cdot,\tilde u)=0,\quad \mbox{on}\   \partial D.
\end{split}
\right.
\end{eqnarray*}
Therefore, $u=\tilde u+2G$ is a solution of PDE $(\ref{linear PDE})$.
\vspace{2mm}

\textbf{Uniqueness:}
If $u'$ is another weak solution of PDE $(\ref{linear PDE})$, it is easy to check that $((u'+2G)(X_t), \nabla (u'+2G)(X_t))$ is another solution of BSDE $(\ref{bsde})$.  By the uniqueness of the solution for this BSDE, $u=u'$ is obtained. 
\end{proof}

Now we come to the probabilistic interpretation of the solution.
\begin{theorem}\label{thmprobainterpretation}

 If $u$ is the weak solution of Neumann boundary problem \eqref{linear PDE},  the process $u(X_t)$ satisfies the following differential equation, for $0\leq s\leq t$,
\begin{equation}\label{decomposition u}\begin{split}
u(X_t)-u(X_s)&=-\int_s^tq(X_r)u(X_r)+f(X_r,u(X_r),\nabla u(X_r))dr-\int_s^t \langle \nabla u(X_r),\vec n\rangle dL_r\\
&-\int_s^t g(X_r)\ast dX_r+\int_s^t\langle \nabla u(X_r), dM_r\rangle.
\end{split}\end{equation}
\end{theorem}

\begin{proof}
 If $u$ is the weak solution of PDE $(\ref{linear PDE})$, then $\tilde{u}=u-2G$ is the weak solution of $(\ref{tran. linear PDE})$.  It is obtained that
\begin{equation*}\begin{split}
\tilde u(X_t)-\tilde u(X_s)=&-\int_s^t \tilde f(X_r,\tilde u(X_r),\nabla \tilde u(X_r))dr-\int_s^t q(X_r)\tilde u(X_r)dr-\int_s^t \tilde h(X_r,\tilde u(X_r)) dL_r\\
&+\int_s^t \langle \nabla \tilde u(X_r), dM_r\rangle.
\end{split}\end{equation*}
By Proposition \ref{decompositionofG}, 
\begin{equation*}\begin{split}
u(X_t)-u(X_s)=&\tilde u(X_t)-\tilde u(X_s)+2G(X_t)-2G(X_s)\\
=&-\int_s^t \tilde f(X_r,\tilde u(X_r),\nabla \tilde u(X_r))dr-\int_s^t q(X_r)\tilde u(X_r)dr-\int_s^t \tilde h(X_r,\tilde u(X_r)) dL_r\\
&+\int_s^t \langle \nabla \tilde u(X_r), dM_r\rangle
+2\int_s^t \langle \nabla G(X_r),dM_r\rangle+2\int_s^t \langle b,\nabla G\rangle(X_r)dr\\
&+2\int_s^t \frac{\partial G}{\partial \vec n}(X_r)dL_r-\int_s^t \nabla G\ast dX_r\\
=&\int_s^t\langle \nabla u(X_r), dM_r\rangle -\int_s^t q(X_r)u(X_r)+G(X_r)+f(X_r,u(X_r),\nabla u(X_r))dr\\
&-\int_s^t \nabla G\ast dX_r+\int_s^t 2 \frac{\partial G}{\partial \vec n}(X_r)-h(X_r, u(X_r))dL_r.
\end{split}\end{equation*}

Noting that $div(g-\nabla G)=-G$ and by Lemma \ref{div g}, we have
$$\int_s^t G(X_r)dr=\int_s^t (g-\nabla G) \ast dX_r.$$
Therefore, $(\ref{decomposition u})$ is proved.
\end{proof}

\section{PDEs with Nonlinear Divergence Terms}

In this section, we assume the divergence term $g$ satisfies Lipschitz condition {\bf (vi)}. 
Let us consider the Picard sequence $(u^n)_{n\geq1}$ defined by $u^0=0$ and for all $n\in \mathbb{N}^*$ we denote by $u^{n}$ the solution of the linear PDE:
\begin{equation}
\label{appro.PDE}
\left\{\begin{split}
&\frac{1}{2}\Delta u^n+qu^n+\langle b,\nabla u^n\rangle-divg(\cdot,u^{n-1},\nabla u^{n-1})+f(\cdot,u^{n},\nabla u^{n})=0,\quad \mbox{on}\ D\,;\\
&\langle \nabla u^n-2g(\cdot,u^{n-1},\nabla u^{n-1}),\vec{n}\rangle+h(\cdot,u^{n})=0, \quad \mbox{on}\  \partial D.
\end{split}\right.
\end{equation}
From Section \ref{linear.case}, we know that $u^n$ exists uniquely.

\subsection{Analytic method}

\begin{theorem}
Under conditions \textbf{(C.1)-(C.3)} and \textbf{(i)-(vii)} with $k<\frac{1}{2}$ and $\alpha$ large enough, assume also $q$ is bounded, then Neumann problem \eqref{PDE} admits a unique weak solution. 
\end{theorem}
\begin{proof}
For simplicity, we set $f_n(x)=f(x,u^n(x),\nabla u^n(x))$, $h_{n}(x)=h(x,u^{n}(x))$ and $g_n(x)=g(x,u^n(x),\nabla u^{n}(x))$. By further calculation, we get
\begin{equation*}\begin{split}
\|\nabla (u^{n+1}-u^n)\|^2=&\,2(\langle b,\nabla (u^{n+1}-u^n)\rangle,u^{n+1}-u^n)+2(g_n-g_{n-1},\nabla (u^{n+1}-u^n))\\
&+2(q(u^{n+1}-u^n),(u^{n+1}-u^n))+2(f_{n+1}-f_n,u^{n+1}-u^n)\\&-\int_{\partial D}(h_{n+1}-h_n)(u^{n+1}-u^n) d\sigma \\
\leq&\,M^2\epsilon_1 \|\nabla (u^{n+1}-u^n)\|^2+\frac{1}{\epsilon_1} \|u^{n+1}-u^n\|^2+\frac{2k^2}{\epsilon_2}\|u^n-u^{n-1}\|^2_{H^1}\\
&+\epsilon_2 \|\nabla (u^{n+1}-u^n)\|^2+2(-\alpha+M_1) \|u^{n+1}-u^n\|^2+\epsilon_3 \|\nabla (u^{n+1}-u^n)\|^2\\
&+\frac{K^2}{\epsilon_3} \|u^{n+1}-u^{n}\|^2+2\beta'\int_{\partial D}|u^{n+1}-u^n|^2d\sigma,
\end{split}\end{equation*}
where $M=\sup_x|b(x)|$, $M_1=\sup_{x\in D} q(x)$. Then
\begin{equation*}\begin{split}
&(1-M^2\epsilon_1-\epsilon_2-\epsilon_3)\|\nabla (u^{n+1}-u^n)\|^2\\
\leq &\, \frac{2k^2}{\epsilon_2}\|u^n-u^{n-1}\|^2_{H^1}
+(2(-\alpha+M_1+\beta'\|Tr\|) +\frac{1}{\epsilon_1}+\frac{K^2}{\epsilon_3})\|u^{n+1}-u^n\|^2.
\end{split}\end{equation*}
Hence, 
\begin{equation*}\begin{split}
(1-M^2\epsilon_1-\epsilon_2-\epsilon_3)&\|u^{n+1}-u^n\|^2_{H^1}
\leq \frac{2k^2}{\epsilon_2}\|u^n-u^{n-1}\|^2_{H^1}
\\&+(1-M^2\epsilon_1-\epsilon_2-\epsilon_3+2(-\alpha+M_1+\beta'\|Tr\|) +\frac{1}{\epsilon_1}+\frac{K^2}{\epsilon_3})\|u^{n+1}-u^n\|^2.
\end{split}\end{equation*}
As $k<\frac{1}{2\sqrt{2}}$ and $\alpha$ large enough (i.e. $-\alpha+M_1+\beta'\|Tr\|+\frac{3}{4}+\sqrt{(-\alpha+M_1+\beta'\|Tr\| +\frac{3}{4})^2+6M^2}<3(\frac{1}{2}-\sqrt{2}k)$), then we can choose $\epsilon_1,\epsilon_2,\epsilon_3>0$ such that
\begin{equation*}\begin{split}
\label{epsilon condition}
&1-M^2\epsilon_1-\epsilon_2-\epsilon_3>0, \\
&1-M^2\epsilon_1-\epsilon_2-\epsilon_3-2\alpha+2M_1+2\beta'\|Tr\|+\frac{1}{\epsilon_1}+\frac{K^2}{\epsilon_3}<0,\\
&2k^2<\epsilon_2(1-M^2\epsilon_1-\epsilon_2-\epsilon_3),
\end{split}\end{equation*}
(For example, one can take $M^2\epsilon_1=\epsilon_3=\frac{1}{3}(-\alpha+M_1+\beta'\|Tr\|+\frac{3}{4}+\sqrt{(-\alpha+M_1+\frac{3}{4})^2+6M^2}), \epsilon_2=\frac{M^2\epsilon_1+\epsilon_3-1}{2}$.)

Set $\gamma =\frac{k^2}{\epsilon_2(1-M^2\epsilon_1-\epsilon_2-\epsilon_3)}<1$. It follows that
\begin{eqnarray*}
\|u^{n+1}-u^n\|^2_{H^1}\leq \gamma \|u^{n}-u^{n-1}\|^2_{H^1}\leq\cdots\leq \gamma^{n}\|u^{1}\|^2_{H^1}.
\end{eqnarray*}
Therefore, $\{u^n\}$ is a Cauchy sequence in $H^1(D)$, and we denote the limit of  $\{u^n\}$ by $u$. For any test function $\phi\in C^{\infty}(D)$, we have
\begin{equation*}\begin{split}
&\frac{1}{2}\int_D \langle \nabla u,\nabla \phi\rangle(x)dx-\int_D \langle b,\nabla u\rangle(x)\phi(x)dx+\int_D q(x)u(x)dx\\
=&\lim_{n\rightarrow\infty}\left(\frac{1}{2}\int_D \langle \nabla u^n,\nabla \phi\rangle(x)dx-\int_D \langle b,\nabla u^n\rangle(x)\phi(x)dx+\int_D q(x)u^n(x)dx\right)\\
=&\lim_{n\rightarrow\infty}\left(\int_D f_n(x)\phi(x)dx
+\int_{D}\langle g_{n-1}(\cdot),\nabla \phi\rangle(x)dx
-\frac{1}{2}\int_{\partial D} h_n(x)\phi(x)d\sigma(x)\right)\\
=&\int_D f(x,u,\nabla u)\phi(x)dx
+\int_{D}\langle g(\cdot,u,\nabla u),\nabla \phi\rangle(x)dx
-\frac{1}{2}\int_{\partial D} h(x,u)\phi(x)d\sigma(x),
\end{split}\end{equation*}
where the last equality is due to
\begin{equation*}\begin{split}
&|f(x,u^n,\nabla u^n)-f(x,u,\nabla u)|\leq |\alpha| |u-u^n|+K|\nabla u-\nabla u^n|,\\
&|g(x,u^{n-1},\nabla u^{n-1})-g(x,u,\nabla u)|\leq k(|u-u^{n-1}|+|\nabla u-\nabla u^{n-1}|),
\end{split}\end{equation*}
and
\begin{equation*}\begin{split}
\int_{\partial D} |h(x,u^n)-h(x,u^{n-1})||\phi(x)|d\sigma(x)&\leq (\int_{\partial D}|\phi(x)|^2d\sigma)^{\frac{1}{2}}\cdot\|h(\cdot,u^n)-h(\cdot,u^{n-1})\|_{L^2(\partial D)}\\
&\leq(\int_{\partial D}|\phi(x)|^2d\sigma)^{\frac{1}{2}}\cdot \|Tr\|\cdot\|h(\cdot,u^n)-h(\cdot,u^{n-1})\|\\
&\leq|\beta|(\int_{\partial D}|\phi(x)|^2d\sigma)^{\frac{1}{2}}\cdot \|Tr\|\cdot \|u^n-u^{n-1}\|\\&\rightarrow 0,\ n\rightarrow \infty.    
\end{split}\end{equation*}
Therefore, $u\in H^1(D)$ is a weak solution for \eqref{PDE}.

\vspace{2mm}
Suppose $u,\bar u$ are two weak solutions,  we obtain that 
\begin{equation*}\begin{split}
\|\nabla (u-\bar u)\|^2=&\,2(\langle b,\nabla (u-\bar u)\rangle,u-\bar u)+2(g(\cdot,u,\nabla u)-g(\cdot,\bar u,\nabla \bar u),\nabla (u-\bar u))\\
&+2(q(u-\bar u),(u-\bar u))+2(f(\cdot, u,\nabla u)-f(\cdot,\bar u,\nabla \bar u),u-\bar u)\\
&-\int_{\partial D}(h(x,u)-h(x,\bar u))(u-\bar u) d\sigma(x) .
\end{split}\end{equation*}
By the same method in the proof of existence, there is a constant $\gamma<1$ such that
$$
\|u-\bar u\|^2_{H^1}\leq \gamma \|u-\bar u\|^2_{H^1},
$$
which implies that $u=\bar u$.
\end{proof}


\subsection{Probabilistic Method}\label{probabilisticapproximation}

In this section, we simply assume $b=0$ in PDE \eqref{PDE},  and consider the symmetric reflecting diffusions correspondingly. Actually, we can combine the drift term $\langle b,\nabla u\rangle$  and nonlinear term $f(x,u,\nabla u)$  into a new nonlinear term $F(x,u,\nabla u):=\langle b,\nabla u\rangle+f(x,u,\nabla u)$ so that this assumption is realized,  without weakening our result.  The solution for the nonlinear PDE will be given by probabilistic
method independent of the analytic one.

Setting $u^0=0$,  we consider the following PDE:
\begin{equation}
\label{appro.PDE2}
\left\{\begin{split}
&\frac{1}{2}\Delta u^n+qu^n-divg(\cdot,u^{n-1},\nabla u^{n-1})+f(\cdot,u^{n},\nabla u^{n})=0,\quad \mbox{on}\ D,\\
&\langle \nabla u^n-2g(\cdot,u^{n-1},\nabla u^{n-1}),\vec{n}\rangle+h(\cdot,u^{n})=0, \quad \mbox{on}\  \partial D.
\end{split}\right.
\end{equation}
By Theorem \ref{theorem linear}, \eqref{appro.PDE2} admits a unique solution $u^n$ for every $n\in \mathbb{N}$.

Let $m$ denote the Lebesgue measure on $D$ and set the pobability space $\Omega'=D\otimes \Omega$  and probability $P^m=m\otimes P$.
$\{X_t\}_{t\geq0}$ is the reflecting Brownian motion in domain  D of the following form:
$$
X_t-X_s=B_t-B_s+\int_s^t \vec{n}(X_r)dL_r, \quad\forall\,0\leq s\leq t.
$$
It is known that, $\{X_t\}_{t\geq0}$ is a symmetric diffusion with initial distribution $m$.

By the symmetricalness, we know that
$$
\bar{B}(s,t)=2X_s-2X_t+B_t-B_s=B_s-B_t-2\int_s^t \vec{n}(X_r)dL_r,
$$
which is a backward martingale under $P^m$ with respect to the backward filtration $\mathcal F'_s=\sigma\{X_r|r\in[s,\infty)\}$.

For $g=(g_1,\cdots,g_N): \mathbb{R}^N\rightarrow \mathbb{R}^N$, as in Section 3 we define the backward stochastic integral as follows
\begin{eqnarray}\label{backwardsibarB}
\int_s^t g_i(X_{r})d\bar{B}^i_t=(L^2-)\lim_{\delta\rightarrow 0}\sum_{j=0}^{n-1} g(X_{t_{j+1}})\bar{B}^i(t_j,t_{j+1}),
\end{eqnarray}
where the limit is over the partition $s=t_0<t_1<\cdots<t_n=t$ and $\delta=\max_j (t_{j+1}-t_j)$.

In this case, for $0\leq s\leq t$, one has
\begin{eqnarray}\label{siofgwrtX}
\int_s^t g\ast dX_r=\int_s^t \langle g(X_r), dB_r\rangle+\int_s^t \langle g(X_r),d\bar{B}_r\rangle+2\int_s^t \langle g,\vec n\rangle(X_r)dL_r.
\end{eqnarray}

The following lemma is from \cite{S}.
\begin{lemma}
(Ito's Formula) Assume that $(Y,Z)$ is a solution of the following BSDE
\begin{eqnarray*}
Y_t=Y_T-\int_t^T \langle Z_r,dB_r\rangle+\int_t^T f(X_r,Y_r,Z_r)dr+\int_t^T h(X_r,Y_r)dL_r+\int_t^T g(X_r,Y_r,Z_r)\ast dX_r.
\end{eqnarray*}
Then, $\forall t\in[0,T]$, one has $P^m-$almost surely, 
\begin{equation*}\begin{split}
|Y_t|^2=&|Y_T|^2-2\int_t^T Y_r\langle Z_r,dB_r\rangle+2\int_t^T Y_rf(X_r,Y_r,Z_r)dr+2\int_t^TY_r h(X_r,Y_r)dL_r\\
&+2\int_t^T Y_r g(X_r,Y_r,Z_r)\ast dX_r-\int_t^T |Z_r|^2dr+2\int_t^T \langle g(X_r,Y_r,Z_r),Z_r\rangle dr.
\end{split}\end{equation*}
\end{lemma}



\begin{theorem}\label{main.theorem}
Under assumptions \textbf{(C.1)-(C.3)} and  \textbf{(i)-(vii)}  with $k<\frac{1}{2\sqrt2}$ and $\alpha$ large enough, 
then Neumann problem \eqref{PDE} admits a unique bounded weak solution. 
\end{theorem}
 
\begin{proof}{\bf Existence}: As $u^n$ is the solution of $(\ref{appro.PDE})$, by Theorem \ref{thmprobainterpretation}, we have
\begin{equation*}\begin{split}
u^n(X_t)-u^n(X_s)=&\int_{s}^t \langle \nabla u^n(X_r),dB_r\rangle-\int_s^t q(X_r)u^n(X_r)dr-\int_s^t f(X_r,u^n(X_r),\nabla u^n(X_r))dr\\
&+\int_s^t\langle \nabla u^n,\vec{n}\rangle (X_r)dL_r-\int_s^t g(X_r,u^{n-1}(X_r),\nabla u^{n-1}(X_r))\ast dX_r\\
=&\int_{s}^t \langle \nabla u^n(X_r),dB_r\rangle-\int_s^t q(X_r)u^n(X_r)dr-\int_s^t f(X_r,u^n(X_r),\nabla u^n(X_r))dr\\
&-\int_{s}^t h(X_r,u^n(X_r))dL_r-\int_s^t \langle g(X_r,u^{n-1}(X_r),\nabla u^{n-1}(X_r)), dB_r+d\bar{B}_r\rangle.
\end{split}\end{equation*}
Applying It\^o's formula to $e^{2\int_0^tq(X_r)dr+\lambda t+\mu L_t}|(u^{n+1}-u^n)(X_t)|^2$, we get
\begin{equation*}\begin{split}
&e^{2\int_0^tq(X_r)dr+\lambda t+\mu L_t}|(u^{n+1}-u^n)(X_t)|^2+\int_t^Te^{2\int_0^rq(X_s)ds+\lambda r+\mu L_r}|\nabla (u^{n+1}-u^n)(X_r)|^2dr\\
=& e^{2\int_0^Tq(X_r)dr+\lambda T+\mu L_T}|(u^{n+1}-u^n)(X_T)|^2-\int_t^Te^{2\int_0^rq(X_s)ds+\lambda r+\mu L_r}|(u^{n+1}-u^n)(X_r)|^2(\lambda dr+\mu dL_r)\\
-&2\int_t^T e^{2\int_0^rq(X_s)ds+\lambda r +\mu L_r}(u^{n+1}(X_r)-u^n(X_r))\langle \nabla (u^{n+1}-u^n)(X_r),dB_r\rangle\\
+&2\int_t^T e^{2\int_0^rq(X_s)ds+\lambda r+\mu L_r}(u^{n+1}(X_r)-u^n(X_r))(f_{n+1}(X_r)-f_n(X_r))dr\\
+&2\int_t^Te^{2\int_0^rq(X_s)ds+\lambda r+\mu L_r} (u^{n+1}(X_r)-u^n(X_r))(h_{n+1}(X_r)-h_n(X_r))dL_r\\
+&2\int_t^T e^{2\int_0^rq(X_s)ds+\lambda r+\mu L_r } \langle \nabla u^{n+1}-\nabla u^n,g_n-g_{n-1}\rangle(X_r)dr\\
+&2\int_t^T e^{2\int_0^rq(X_s)ds+\lambda r+\mu L_r} ((u^{n+1}-u^{n})(X_r)) \langle g(X_r,u^{n-1}(X_r),\nabla u^{n-1}(X_r)), dB_r+d\bar{B}_r\rangle.
\end{split}\end{equation*}
Taking expectation $E^m$ on both sides of the above equality, and letting $T$ tends to infinity, since $u^n$ is bounded, we have
$$\lim_{t\rightarrow\infty}E^m[e^{2\int_0^tq(X_r)dr}|(u^{n+1}-u^n)(X_t)|^2]=0$$
and
\begin{equation*}\begin{split}
&E^m\Big[e^{2\int_0^tq(X_r)dr+\lambda t+\mu L_t}|(u^{n+1}-u^n)(X_t)|^2+\int_t^\infty e^{2\int_0^rq(X_s)ds+\lambda r+\mu L_r}|\nabla (u^{n+1}-u^n)(X_r)|^2dr\Big]\\
\leq&\, E^m\int_t^\infty e^{2\int_0^rq(X_u)du+\lambda r+\mu L_r} \Big((-\lambda-2\alpha+\frac{K^2}{\epsilon_1})|(u^{n+1}-u^n)(X_r)|^2+\epsilon_1|\nabla (u^{n+1}-u^n)(X_r)|^2\Big)dr\\
&+\epsilon_2 E^m\int_t^\infty e^{2\int_0^rq(X_s)ds+\lambda r+\mu L_r} |\nabla (u^{n+1}-u^n)(X_r)|^2dr\\
&+(-2\beta-\mu) E^m\int_t^\infty e^{2\int_0^rq(X_s)ds+\lambda r+\mu L_r} | (u^{n+1}-u^n)(X_r)|^2dL_r\\
&+\frac{2k^2}{\epsilon_2}E^m\int_t^\infty e^{2\int_0^rq(X_s)ds+\lambda r+\mu L_r}|(u^{n}-u^{n-1})(X_r)|^2+|\nabla (u^{n}-u^{n-1})(X_r)|^2dr.
\end{split}\end{equation*}
By further calculation, we obtain
\begin{equation*}\begin{split}
&E^m\Big[(\lambda+2\alpha-\frac{K^2}{\epsilon_1})\int_t^\infty  e^{2\int_0^rq(X_s)ds+\lambda r+\mu L_r} |(u^{n+1}-u^n)|^2(X_r)dr\\
&+(1-\epsilon_1-\epsilon_2)\int_t^\infty e^{2\int_0^rq(X_s)ds+\lambda r+\mu L_r}|\nabla(u^{n+1}-u^n)|^2(X_r)dr\Big]\\
\leq&\, \frac{2k^2}{\epsilon_2}E^m\int_t^\infty e^{2\int_0^rq(X_s)ds+\lambda r+\mu L_r}(|(u^{n}-u^{n-1})(X_r)|^2+|\nabla (u^{n}-u^{n-1})(X_r)|^2)dr.
\end{split}\end{equation*}
Suppose $k<\frac{1}{2\sqrt2}$, and choose $\lambda,\epsilon_1,\epsilon_2$ such that $\lambda+2\alpha-\frac{K^2}{\epsilon_1}=1-\epsilon_1-\epsilon_2$ and
$k^2<\frac{\epsilon_2(1-\epsilon_1-\epsilon_2)}{2}$.

(For example,  set $\epsilon_2=\frac{1-\epsilon_1}{2}$, we can choolse $\epsilon_1\in(0,1)$ such that 
 $$\frac{\epsilon_2(1-\epsilon_1-\epsilon_2)}{2}=\frac{(1-\epsilon_1)^2}{8}\in(k^2,\frac{1}{8}),\ i.e.  \ \ 0<\epsilon_1<1-2\sqrt2 k. $$
 
 Furthermore,  if positive number $\alpha$ large  enough satisfying
 $$
 \alpha>\frac{\sqrt{2}}{2}k+\frac{K^2}{1-2\sqrt{2}k},
 $$
 then we can choose $\lambda <0$  such that $$\lambda:=-2\alpha+\frac{K^2}{\epsilon_1}+1-\epsilon_1-\epsilon_2=-2\alpha+\frac{K^2}{\epsilon_1}+\frac{1-\epsilon_1}{2}<0$$
 It also satisfies $\lambda>-2\alpha+K^2$.)

Set $\gamma=\frac{2k^2}{\epsilon_2(1-\epsilon_1-\epsilon_2)}$, then it follows
\begin{equation*}\begin{split}
&E^m\int_t^\infty e^{2\int_0^rq(X_s)ds+\lambda r+\mu L_r} (|(u^{n+1}-u^n)(X_r)|^2+ |\nabla(u^{n+1}-u^n)(X_r)|^2)dr\\
\leq&\, \gamma E^m\int_t^\infty e^{2\int_0^rq(X_s)ds+\lambda r+\mu L_r}( |(u^{n}-u^{n-1})|^2(X_r)+ |\nabla(u^{n}-u^{n-1})|^2(X_r))dr\\
\leq& \cdots\leq \gamma^n E^m\int_t^\infty e^{2\int_0^rq(X_s)ds+\lambda r+\mu L_r}( |u^1|^2(X_r)+ |\nabla u^1|^2(X_r))dr.
\end{split}\end{equation*}

By standard calculation,  since
\begin{equation*}\begin{split}
&e^{2\int_0^tq(X_r)dr+\lambda t+\mu L_t}|u^{1}(X_t)|^2+\int_t^\infty e^{2\int_0^rq(X_s)ds+\lambda r+\mu L_r}|\nabla u^{1}(X_r)|^2dr\\
=& -\int_t^Te^{2\int_0^rq(X_s)ds+\lambda r+\mu L_r}|u^{1}(X_r)|^2(\lambda dr+\mu dL_r)
-2\int_t^T e^{2\int_0^rq(X_s)ds+\lambda r +\mu L_r}u^{1}(X_r)\langle \nabla u^{1}(X_r),dB_r\rangle\\
+&2\int_t^T e^{2\int_0^rq(X_s)ds+\lambda r+\mu L_r}u^{1}(X_r)f_1(X_r)dr
+2\int_t^Te^{2\int_0^rq(X_s)ds+\lambda r+\mu L_r} u^{1}(X_r)h_1(X_r)dL_r\\
+&2\int_t^T e^{2\int_0^rq(X_s)ds+\lambda r+\mu L_r } \langle \nabla u^{1},g_0\rangle(X_r)dr
+2\int_t^T e^{2\int_0^rq(X_s)ds+\lambda r+\mu L_r} u^{1}(X_r)g_0(X_r)(dB_r+d\bar{B}_r),
\end{split}\end{equation*}
then, by the boundedness of $u^1,f,h,g$, we obtain
\begin{equation*}\begin{split}
&E^m\Big[(\lambda+2\alpha-\frac{K^2}{\epsilon_1})\int_t^\infty  e^{2\int_0^rq(X_s)ds+\lambda r+\mu L_r} |u^1(X_r)|^2dr\\&
\quad\quad+(1-\epsilon_1-\epsilon_2)\int_t^\infty e^{2\int_0^rq(X_u)du+\lambda r+\mu L_r}|\nabla u^1(X_r)|^2dr\Big]\\
\leq& E^m\Big[\int_t^\infty e^{2\int_0^rq(X_s)ds+\lambda r+\mu L_r}\big((|u^1(X_r) f_0(X_r)|+\frac{1}{\epsilon_2} |g_0(X_r)|^2)dr+|u^1(X_r) h_0(X_r)|dL_r\big)\Big].
\end{split}\end{equation*}
Therefore, $(e^{\int_0^tq(X_s)ds+\frac{1}{2}\lambda t+\frac{1}{2}\mu L_t}u^n(X_t),e^{2\int_0^tq(X_s)ds+\frac{1}{2}\lambda t+\frac{1}{2}\mu L_t}\nabla u^n(X_t))$ is a Cauchy Sequence in $L^2(\Omega\times[0,\infty))$. We denote the limit as $(\tilde Y_t,\tilde Z_t)$. Set
$$Y_t=e^{-\int_0^tq(X_s)ds-\frac{1}{2}\lambda t-\frac{1}{2}\mu L_t}\tilde{Y}_t,\ \ Z_t=e^{-\int_0^tq(X_s)ds-\frac{1}{2}\lambda t-\frac{1}{2}\mu L_t}\tilde{Z}_t.$$
Then, it is easy to check that $(Y,Z)$ is the solution of the following BSDE
\begin{equation}
\label{BSDE}
\left\{\begin{split}
&Y_t =Y_T+\int_{t}^T\langle Z_r,dB_r\rangle-\int_t^T q(X_r)Y_rdr-\int_t^T f(X_r,Y_r,Z_r)dr+\int_t^T \langle Z_r,\vec{n}\rangle dL_r\\
&\quad\quad-\int_t^T g(X_r,Y_r,Z_r)\ast dX_r, \\
&\lim_{t\rightarrow \infty} e^{\int_0^tq(X_r)dr+\frac{1}{2}\lambda t+\frac{1}{2}\mu L_t}Y_t=0,\ \mbox{in}\ L^2(\Omega,P^x).
\end{split}\right.
\end{equation}
Set $u_0=E^x[Y_0]$ and $v_0=E^x[Z_0]$. Consider the following linear equation
\begin{equation*}
\left\{
\begin{split}
&\frac{1}{2}\Delta \tilde u+q\tilde u-div(g(\cdot,u_0,v_0))+f(\cdot,u_0,v_0)=0,\quad \mbox{on} \  D,\\
&\langle \nabla \tilde u-2 g(\cdot,u_0,v_0), \vec{n}\rangle+h(\cdot,u_0)=0,\quad \mbox{on}\   \partial D,
\end{split}
\right.
\end{equation*}
and by the same method in Section \ref{linear.case}, we find that $u_0=\tilde u$ and $v_0=\nabla \tilde u$. Therefore, $\tilde u$ is a weak solution for PDE $(\ref{PDE})$.   

\textbf{Uniqueness}: Suppose that $\bar u$ is another weak solution of PDE $(\ref{PDE})$.  By the same method in Theorem \ref{thmprobainterpretation} (set $\bar{g}(x)=g(x,\bar{u}(x),\nabla \bar{u}(x))$), we find $\bar{Y}_t=\bar u(X_t)$, $\bar Z_t=\nabla \bar u(X_t)$ satisfies the following BSDE

\begin{equation*}
\left\{\begin{split}
&\bar Y_t =\bar Y_T+\int_{t}^T\langle\bar Z_r,dB_r\rangle-\int_t^T q(X_r)\bar Y_rdr-\int_t^T f(X_r,\bar Y_r,\bar Z_r)dr+\int_t^T \langle \bar Z_r,\vec{n}\rangle dL_r\\
&\quad\quad-\int_t^T g(X_r,\bar Y_r,\bar Z_r)\ast dX_r,\\
&\lim_{t\rightarrow \infty} e^{\int_0^tq(X_r)dr+\frac{1}{2}\lambda t+\frac{1}{2}\mu L_t}\bar Y_t=0,\ \mbox{in}\ L^2(\Omega,P^x).
\end{split}\right.
\end{equation*}
By Ito's formula, we get
\begin{equation*}\begin{split}
&e^{2\int_0^tq(X_r)dr+\lambda t+\mu L_t}|Y_t-\bar Y_t|^2+\int_t^Te^{2\int_0^rq(X_s)ds+\lambda r+\mu L_r}|Z_r-\bar Z_r|^2dr\\
=& -\int_t^\infty e^{2\int_0^rq(X_s)ds+\lambda r+\mu L_r}|Y_r-\bar  Y_r|^2(\lambda dr+\mu dL_r)\\
-&2\int_t^T e^{2\int_0^rq(X_s)ds+\lambda r +\mu L_r}(Y_r-\bar Y_r)\langle Z_r-\bar Z_r,dB_r\rangle\\
+&2\int_t^T e^{2\int_0^rq(X_s)ds+\lambda r+\mu L_r}(Y_r-\bar Y_r)(f(X_r,Y_r,Z_r)-f(X_r,\bar Y_r,\bar Z_r))dr\\
+&2\int_t^Te^{2\int_0^rq(X_s)ds+\lambda r+\mu L_r} (Y_r-\bar Y_r)(h(X_r,Y_r))-h(X_r,\bar Y_r)dL_r\\
+&2\int_t^T e^{2\int_0^rq(X_s)ds+\lambda r+\mu L_r } \langle Z_r-\bar Z_r, g(X_r,Y_r,Z_r)-g(X_r,\bar Y_r,\bar Z_r)\rangle dr\\
+&2\int_t^T e^{2\int_0^rq(X_s)ds+\lambda r+\mu L_r} (Y_r-\bar  Y_r)\langle g(X_r,Y_r,Z_r)-g(X_r,\bar Y_r,\bar Z_r),dB_r+d\bar{B}_r\rangle.
\end{split}\end{equation*}
By the properties of $f,g,h$, it follows that
\begin{equation*}
\begin{split}
E^m[&(\lambda +2\alpha-\frac{K^2}{\epsilon_1}-\frac{2k^2}{\epsilon_2})\int_t^\infty e^{2\int_0^rq(X_s)ds+\lambda r+\mu L_r}|Y_r-\bar  Y_r|^2\\
&+(1-\epsilon_1-\epsilon_2-\frac{2k^2}{\epsilon_2})\int_t^\infty e^{2\int_0^rq(X_s)ds+\lambda r+\mu L_r}|Z_r-\bar  Z_r|^2]	\leq 0.
\end{split}
\end{equation*}
As stated before, we choose $\lambda,\epsilon_1,\epsilon_2$ such that $\lambda-2\alpha-\frac{K^2}{\epsilon_1}=1-\epsilon_1-\epsilon_2$ and
$k^2<\frac{\epsilon_2(1-\epsilon_1-\epsilon_2)}{2}$, then we obtain
$$
E^m[\int_t^\infty e^{2\int_0^rq(X_s)ds+\lambda r+\mu L_r}|Y_r-\bar  Y_r|^2+|Z_r-\bar  Z_r|^2]	\leq 0.
$$
This implies that $|Y_t-\bar Y_t|^2=0$ and $|Z_t-\bar Z_t|^2=0$ for every $t\geq 0$,  which provide that $u=\bar u$ and $\nabla u=\nabla \bar u$.
\end{proof}

The following theorem summerizes the relationship between PDEs and BSDEs. The first part can be proved easily as in Theorem \ref{thmprobainterpretation} and the second part follows the uniqueness in the last theorem.
\begin{theorem}
(1) If $u$ is the weak solution of PDE \eqref{PDE}, then $Y_t=u(X_t)$ and $Z_t=\nabla u(X_t)$ solves BSDE \eqref{BSDE}. \\
(2) Suppose $f,g,h$ and $q$ satisfy conditions \textbf{(C.1)-(C.3)} and \textbf{(i)-(vi)}, if $(Y,Z)$ is the solution of BSDE  \eqref{BSDE}, then $u(x)=E^x[Y_0]$ is the weak solution of PDE \eqref{PDE}.
\end{theorem}


\newpage
\begin{thebibliography}{99}

\bibitem{BH}  R. F. Bass and P. Hsu,
\textit{Some potential theory for reflecting Brownian motion in H\"{o}lder and Lipschitz domains},
 Ann. Prob. \textbf{19} (1991), 486-508.

\bibitem{CZ} Z. Q. Chen and T. S. Zhang,
\textit{ Time-reversal and elliptic boundary value problems},
Ann. Prob. \textbf{37} (2009), 1008-1043.

\bibitem{CZ2} Z. Q. Chen and T. S. Zhang,
\textit{ A probabilistic approach to mixed boundary value problems for elliptic operators with singular coefficients}, 
Proceeding of the American Mathematical Society \textbf{142(6)} (2014), 2135-2149.
 
\bibitem{DP} R. W. R. Darling and E. Pardoux,
\textit{Backwards SDE with random terminal time and applications to semilinear elliptic PDE},
Ann. Prob. \textbf{25} (1997), 1135-1159.

\bibitem{Gernard} W. D. Gerhard,
\textit{The probabilistic solution of the Dirichlet problem for $\frac{1}{2}\nabla+\langle a,\nabla \rangle+b$ with singular coefficients},
J. Theoret. Proba. \textbf{5} (1992), 503-520.
 
\bibitem{GT} D.Gilbarg and N.S.Trudinger, 
\textit{Elliptic Partial Differential Equations of Second Order Reprint of Second Order, reprint of the 1998 Edition.}
Grundlehren Der Mathematischen Wissenschaften {\bf 224(3)} (2001), 469-484. 

\bibitem{H} Y. Hu,
\textit{Probabilistic interpretation of a system of quasilinear elliptic partial differential equations under Neumann boundary conditions},
Stoch. Proc. Appl. \textbf{48} (1993), 107-121.    
    
\bibitem{PEI} P. Hsu,
\textit{Probabilistic Approach to the Neumann Problem}, Comm. Pure and Appl. Math. \textbf{38(4)} (1985), 445-472.    

\bibitem{K} S. Kakutani,
\textit{Two-dimensional Brownian motion and harmonic functions},
Proc. Imp. Acad. (Tokyo) \textbf{20} (1944), 706-714.

\bibitem{LS} P. L. Lions and A. S. Sznitman,
\textit{Stochastic differential equations with reflecting boundary conditions},
Comm. Pure and Appl. Math. \textbf{37} (1984), 511-537.

\bibitem{PZ} E. Pardoux and S. G. Zhang,
\textit{Generalized BSDEs and nonlinear Neumann boundary value problems}, 
Probab. Theory Related Fields {\bf 110} (1998), 535-558.

\bibitem{S} L.Stoica, \textit{A probabilistic interpretaion of the divergence and BSDE's}, 
Stoch. Proc. Appl. {\bf 103} (2003), 31-55.
    
\bibitem{T} N. S. Tr\"{u}dinger,
\textit{Linear elliptic operators with measurable coefficients},
Ann. Scuola Norm. Sup. Pisa \textbf{27} (1973), 265-308.    
   
 \bibitem{YZ} X. Yang and  T. S. Zhang, \textit{Mixed Boundary Value  Problems of  Semilinear Elliptic PDEs  and BSDEs with Singular Coefficients}, 
 Stoch. Proc. Appl. {\bf124} (2014), 2442-2478. 
 
\bibitem{Z} T. S. Zhang,
\textit{A probabilistic approach to Dirichlet problems of semilinear elliptic PDEs with singular coefficients},
Ann.Prob. \textbf{39} (2011),1502-1527.

\end{thebibliography}
\end{document}